\documentclass[11pt]{amsart}
\usepackage{amsmath,amsthm,amssymb,hyperref,fancyhdr,mathtools}
\usepackage{derivative}
\usepackage{color}
\usepackage{comment}
\usepackage{enumitem}

\theoremstyle{plain}
\newtheorem{theorem}{Theorem}
\newtheorem{lemma}[theorem]{Lemma}
\newtheorem{proposition}[theorem]{Proposition}

\newtheorem{definition}[theorem]{Definition}

\numberwithin{theorem}{section}
\numberwithin{equation}{section}

\newcommand{\ip}[2]{\ensuremath{\left\langle#1,#2\right\rangle}}

\newcommand{\B}{\mathbb{B}}
\newcommand{\C}{\mathbb{C}}

\newcommand{\cA}{\mathcal{A}}

\newcommand{\cL}{\mathcal{L}}
\newcommand{\cS}{\mathcal{S}}

\newcommand{\SL}{\mathrm{SL}}
\newcommand{\SU}{\mathrm{SU}}
\newcommand{\U}{\mathrm{U}}

\newcommand{\id}{\mathrm{id}}

\newcommand{\Berg}{\cA^2(\B^n)} 

\newcommand{\bddf}{L^\infty(\B^n)} 
\newcommand{\bddfG}{L^\infty(G)} 
\newcommand{\Lone}{L^1(\B^n,d\lambda)} 
\newcommand{\LoneG}{L^1(\G)} 

\newcommand{\bdd}{\cL(\cA^2)}

\newcommand{\traceop}{\cS^1(\cA^2)}

\newcommand{\toepalg}{\mathfrak{T}(L^\infty)}

\newcommand{\wloc}{\mathcal{W}_{loc}^\beta(\cA^2)}

\newcommand{\toep}{\mathcal{T}(L^\infty)}
\newcommand{\haar}[1]{\ d\mu_\G(#1)}

\newcommand{\inv}[1]{\ d\lambda(#1)}

\newcommand{\actopL}[2]{L_{#1}(#2)}
\newcommand{\actfL}[2]{\ell_{#1}(#2)}

\newcommand{\act}[2]{\tau_{#1}(#2)}

\newcommand{\G}{G}
\newcommand{\Br}{\boldsymbol{B_r}}
\newcommand{\Dr}{D_r}

\newcommand{\I}[1]{I_{#1}^\beta}


\author{Vishwa Dewage }
\author{Mishko Mitkovski }
\thanks{M.~M. was supported in part by NSF grant DMS-2000236.}

\address{Department of Mathematics, Clemson University\\
South Carolina, SC 29634, USA\\
E-Mail Dewage:vdewage@clemson.edu
E-Mail Mitkovski: mmitkov@clemson.edu}

\title[Density of Toeplitz operators in the Toeplitz algebra]{On the density of Toeplitz operators in the Toeplitz algebra over the Bergman space of the unit ball}

\begin{document}


\begin{abstract}
    We use quantum harmonic analysis and representation theory to provide a new proof of Xia's theorem: "Toeplitz operators are norm dense in the Toeplitz algebra over the Bergman space of the unit ball."
\end{abstract}

\maketitle


\section{Introduction}
In quantum harmonic analysis, we discuss classical harmonic analysis notions such as convolutions and Fourrier transforms for operators.
Currently, Werner's quantum harmonic analysis (QHA) \cite{W84}, is being used effectively in many areas of mathematics including time-frequency analysis, mathematical physics and operator theory \cite{BBLS22, FHL24, KL21, KLSW12, LS20, S24}. QHA was found to be well-suited to study Toeplitz operators on Fock spaces \cite{DM23, DM24, DM24a, FH23, F19, FR23}.

In \cite{X15}, Xia proves that the Toeplitz operators are norm-dense in the Toeplitz algebra over both the Bergman space and the Fock space. He utilizes the notion of weakly localized operators introduced in \cite{IMW13} for this purpose. In \cite{F19}, among other things, Fulsche uses QHA to reprove this density result. As discussed in \cite{DDMO24}, QHA on the Bergman space over the unit ball, is more challenging due to the noncommutativity of the underlying group. However, the available theory is still useful to study Toeplitz operators. In this note, we employ QHA on Bergman spaces to simplify and clarify certain aspects of Xia's proof of the density result for the Bergman space.

This article is organized as follows. In Section \ref{sec:prelim}, we recall some facts on the Bergman space, on representation theory, and on quantum harmonic analysis. In Section \ref{sec:wloc}, we present our proof of Xia's theorem: "Toeplitz operators are norm-dense in the Toeplitz algebra over the Bergman space."

\section{Preliminaries}\label{sec:prelim}

The Bergman space $\Berg$ is the space of all holomorphic functions that are square-integrable w.r.t. the normalized Lebesgue measure $dz$. It is well known that $\Berg$ is a reproducing kernel Hilbert space with reproducing kernel
\[
   K(w,z)=K_z(w)=\frac{1}{(1-\Bar{z}w)^{n+1}},\ \ z,w\in \B^n.
\]
We have the reproducing formula:
$$f(z)=\ip{f}{K_z},\quad z\in\B^n, f\in \Berg.$$
Note that the normalized reproducing kernels $k_z$ are given by
$$k_z(w)=\frac{K_z(w)}{\|K_z\|}=\frac{(1-|z|^2)^{(n+1)/2}}{(1-\Bar{z}w)^{n+1}},\ \ z,w\in \B^n.$$

The Bergman space $\Berg$ is a closed subspace of $L^2(\B^n,dz)$ and the Bergman projection $P:L^2(\B^n,dz)\to \Berg$ is given by
\[
    (Pf)(z)=\ip{f}{K_z},\ \ z\in \B^n,\ \ f\in L^2(\B^n,dz).
\]

\subsection{Toeplitz operators and Berezin transforms}
We denote the algebra of bounded operators on $\Berg$ by $\bdd$ and the algebra of all trace-class operators on $\Berg$ is denoted by $\traceop$.
Given a function $a:\C^n\to \C$ the Toeplitz operator $T_a$ with symbol $a$ is defined formally by 
\[
   T_af=P(af),\ \ f\in \Berg.
\]
If $a\in \bddf$ then $T_a\in \bdd$   with  $\|T_a\|\leq \|a\|_\infty$. 
The \textit{Toeplitz algebra} $\toepalg$ 
is the $C^*$-subalgebra of $\bdd$  generated by the Toeplitz operators with symbols in $\bddf$.

\subsection{The representation of \texorpdfstring{$\SU(n,1)$}{SU(n,1)} on the Bergman space}
Denote by $\SU (n,1)$ the elements in $\SL(n+1,\C)$ that preserve the hermitian form $(z,w)_{n,1} = z_1\overline{w_1} + \cdots + z_n\overline{w_n} - z_{n+1}\overline{w_{n+1}}$. 
Throughout this article, we write $\G=\SU(n,1)$ and  $K=\U(n)$.
We identify $K$ as a subgroup of $G$ by the map 
\[A\mapsto \Tilde{A}=\begin{pmatrix} A & 0\\ 0 & 1/\det A\end{pmatrix}.\] 
An element $g=\begin{pmatrix} A & u \\ v^* & d \end{pmatrix}\in G$, where $A\in M(n,\C)$, $d\in \C$, and $u,v\in \C^n$, acts on $\B^n$ by the fractional linear transformations:
\begin{equation}\label{def:Action}
    \left(\begin{matrix} A & u \\ v^* & d \end{matrix}\right) \cdot z := \frac{Az + u}{v^* z + d} = \frac{1}{\langle z, v\rangle +d} (Az+u), \quad z\in \B^n, \,\,  \begin{pmatrix} A & u \\ v^* & d \end{pmatrix}\in G .
\end{equation}
The group $G$ acts transitively on $\B^n$ by this action, and we have $K\cdot 0=0$.
Therefore we identify $G/K=\B^n$ via the map $gK \mapsto g\cdot 0$.

The translation of a function $f:\B^n\to\C$ by $g\in G$ is given by
$$(\actfL{g}{f})(w)=f(g^{-1}\cdot w)$$

The discrete series representation $\pi$ of $G$ acting on $\Berg$ is given by
$$\pi(g)f=j(g^{-1},\cdot)\actfL{g}{f}, \ \ f\in \Berg$$
where the cocycle $j$ is given by
$$j( g, z) = (v^*z + d)^{-(n+1)},\quad
    z\in \B^n,\,\, g= \begin{pmatrix} A & u \\ v^* & d \end{pmatrix}  \in \G.$$
We have the following useful properties of the cocycle $j$ from Lemma 2.1 in \cite{DDMO24}:
\begin{equation}\label{eq:cocyclezero}
|j(g,0)|=(1-|g\cdot 0|^2)^{\frac{1}{2}(n+1)},\quad g\in G   
\end{equation}
and
\begin{equation}
j(g^{-1},\cdot)=\overline{j(g,0)}K_{g\cdot 0},\quad g\in G,
\end{equation}
where $K_{g\cdot 0}$ is a reproducing kernel. Then it follows that
\begin{equation}\label{eq:cocyclerep}
    \pi(g)1=\overline{j(g,0)}K_{g\cdot 0},\quad g\in G,
\end{equation}
where $1$ denotes the constant function one. Clearly,
\begin{equation}\label{eq:normalizedrep}
    k_{g\cdot 0}=|j(g,0)|K_{g\cdot 0},\quad g\in G.
\end{equation}
However, in general, $\pi(g)1\neq k_{g\cdot 0}$.

\subsection{Functions on the ball as functions on the group}
We define the invariant measure on $\B^n$ by
$$d\lambda(z)=\frac{1}{(1-|z|^2)^{n+1}}dz.$$
A function on the ball $f:\B^n\to \C$ can be identified with a function on the group, $\Tilde{f}:\G\to \C$ by setting
$$\Tilde{f}(g)=f(g\cdot 0),\quad g\in G.$$
In particular if $f\in \Lone$, we have that
\begin{equation}\label{eq:integral}
    \int_{G} \Tilde{f}(g) \haar{g}=\int_{\B^n} f(z) \inv{z}
\end{equation}
by taking a suitable normalization $\mu_G$  of the Haar measure on $G$. Then it follows that
\begin{equation}\label{eq:integraldz}
    \int_{G} |j(g,0)|^2 \Tilde{f}(g) \haar{g}=\int_{\B^n} f(z) \ dz.
\end{equation}

\subsection{Schur orthogonality relations}
It is well-known that $\pi$ is a square-integrable irreducible unitary representation. Given $f_1,f_2\in \Berg$, the matrix coefficients $\pi_{f_1,f_2}:G\to\C$ are given by
$$\pi_{f_1,f_2}(g)=\ip{f_1}{\pi(g)f_2},\quad g\in G.$$
Then by the Schur-orthogonality relations we have the following:
$$\ip{\pi_{f_1,f_2}}{\pi_{f_3,f_4}}=\frac{1}{d_\pi}\ip{f_1}{f_3}\ip{f_2}{f_4}$$
where $d_\pi>0$ is called the formal dimension of $\pi$. In the following lemma, we compute $d_\pi$.

\begin{lemma}
    The formal dimension $d_\pi$ is given by $d_\pi=1$.
\end{lemma}

\begin{proof}
    Note that by Schur orthogonality relations 
    \begin{align*}
        \ip{\pi_{1,1}}{\pi_{1,1}}=\frac{1}{d_\pi}\ip{1}{1}\ip{1}{1}
        = \frac{1}{d_\pi}.
    \end{align*}
    Also,
    \begin{align*}
        \ip{\pi_{1,1}}{\pi_{1,1}}&=\int_G \ip{1}{\pi(g)1}\ip{\pi(g)1}{1}\haar{g}\\
        &=\int_G |j(g,0)|^2\ip{1}{K_{g\cdot 0}}\ip{K_{g\cdot 0}}{1}\haar{g} \quad \text{(by \ref{eq:cocyclerep})} \\
        &= \int_{\B^n} \ip{1}{K_z}\ip{K_z}{1} \ dz\ \ \text{(by \ref{eq:integraldz})}\\
        &=\int_{\B^n} 1 \ dz \ \ \text{(by the reproducing formula)}\\
        &=1.
    \end{align*}
    Then $d_\pi=1$ as required.
\end{proof}

\subsection{Quantum harmonic analysis on the Bergman space}
Here we collect a few notions from QHA, to be used later in this article. For a detailed consideration of QHA on the Bergman space, we point to \cite{DDMO24}. First, we recall the usual convolutions of functions.
Given functions $\psi:G\to \C$ and $a:\B^n\to \C$, we define the convolution $\psi\ast a:\B^n\to \C$ formally by
$$(\psi\ast a)(z)=\int_{G} \psi(g)(\actfL{g}{a})(z)\haar{g}=\int_{G} \psi(g)f(g^{-1}\cdot z)\haar{g}.$$
Let $p,q\in[1,\infty]$ s.t. $\frac{1}{p}+\frac{1}{q}=1$. Then by Young's inequality, we have the following:
If $\psi\in L^p(G)$ and $a\in L^q(\B^n,d\lambda)$ then $\psi\ast a\in \bddf$ with\\ 
    $$\|\psi\ast a\|_\infty\leq \|\psi\|_p\|a\|_q.$$

The translation of an operator $S\in \bdd$ by $g\in G$ is given by
$$\actopL{g}{S}=\pi(g)S\pi(g)^*.$$

Now we define the convolution of a function and an operator. 
Given a function $\psi:G\to \C$ and and an operator $S\in \bdd$, we define the QHA convolution $\psi\ast S$ formally by the weak-integral
$$\psi\ast S=\int_G\psi(g) \actopL{g}{S} \haar{g}.$$
When $\psi\ast S$ exists as a bounded operator, we have
$$\ip{(\psi\ast S)f_1}{f_2}=\int_G \psi(g)\ip{\actopL{g}{S}f_1}{f_2}\haar{g},\quad f_1,f_2\in \Berg.$$
Then by the QHA-Young's inequality,
We have the following:
\begin{enumerate}
    \item If $\psi\in \LoneG$ and $S\in \bdd$, then $\psi\ast S\in \bdd$ with\\
    $\|\psi\ast S\|\leq \|\psi\|_1\|S\|$.
    \item If $\psi\in \bddfG$ and $S\in \traceop$, then $\psi\ast S\in \bdd$ with\\
    $\|\psi\ast S\|\leq \|\psi\|_\infty\|S\|_1$.
\end{enumerate}
The first identity above is easy to verify. A proof of the second can be found in \cite{H23} or \cite{DDMO24}.

Let $\Phi$ be the rank one operator $\Phi=1\otimes 1$. Then for $a\in \bddf$, the Toeplitz operator $T_a$ can be written as the QHA convolution \cite[Lemma 4.7]{DDMO24}:
$$T_a=a\ast \Phi.$$

\subsection{Involutions}
For each $w\in\B^n$, let $\tau_w : \B^n\rightarrow \B^n$ be the biholomorphism that interchanges $w$ and $0$. Recall that $\tau_w (0) = w$ and that $\tau_w (w) = 0$. Then $\tau_w$ is an involution, i.e. $\tau_w^{-1}=\tau_w=$.  In the following lemma, we consider the elements of $G$ as elements of the automorphisms group $\mathrm{Aut}(\B_n)$.

The following Lemma follows from \cite[Theorem 1.4]{Z05}. We include its simple proof.

\begin{lemma}\label{lem:tau}
    Let $g\in \G$. Then as elements of $\mathrm{Aut}(\B_n)$, we have that
    $$g=\tau_{g\cdot 0}\circ k_g$$ 
    for some $k_g\in K$.
\end{lemma}
\begin{proof}
Since
$(\tau_{g\cdot 0}\circ g)(0)=\act{g\cdot 0}{g\cdot 0} = 0$, we have that $\tau_{g\cdot 0}\circ g=k_g$ for some $k_g\in K$. Then it follows that
$g=\tau_{g\cdot 0}\circ k_g$.
\end{proof}

\begin{lemma}\label{lem:absolute_value}
    Let $g,h\in G$. Then $|h^{-1}g|=|\act{h\cdot 0}{g\cdot 0}|$
\end{lemma}

\begin{proof}
    By Lemma \ref{lem:tau}, we have $g=\tau_{g\cdot 0}\circ k_g$ and $h=\tau_{h\cdot 0}\circ k_h$ for some $k_g,k_h\in K$. Then
    \begin{align*}
        |h^{-1}g|&=|(k_h\circ \tau_{h\cdot 0}\circ \tau_{g\cdot 0}\circ k_g)(0)|\\
        &=|(k_h\circ\tau_{h\cdot 0})(\tau_{g\cdot 0}(0))|\\
        &=|k_h\cdot\act{h\cdot 0}{g\cdot 0}|\\
        &=|\act{h\cdot 0}{g\cdot 0}|
    \end{align*}
    as the absolute value $|\cdot|$ is rotation invariant.
\end{proof}


\section{Weakly localized operators and the Toeplitz algebra}\label{sec:wloc}

In \cite{IMW13}, the authors introduce the algebra of $\beta$-weakly localized operators $\wloc$. They prove that $\wloc$ is a selfadjoint subalgebra of $\bdd$ containing the Toeplitz operators. Then in \cite{X15}, Xia proved that each weakly localized algebra is contained in the norm closure of the set of all Toeplitz operators with bounded symbols. This implies the well-known result of Xia "Toeplitz operators are norm dense in the Toeplitz algbera".  In this section, we use QHA notions to clarify and simplify Xia's proof. The main advantage of our methods is that we bypass the use of separated sets. However, the core idea of the proof is essentially the same as Xia's.

For $r\in[0,1)$, and $z\in \B^n$, we define the set $\Dr(z)$ by 
$$\Dr(z)=\{w\in\B^n\mid |\act{w}{z}|<r\}.$$

\begin{definition}
    For $S\in \bdd$, $\beta\in (0,1)$ and for $r\in [0,1)$, we define 
    $$\I{S}(r):=\sup_{z\in \B^n} \int_{\Dr(z)^\mathsf{c}} |\ip{SK_z}{K_w}|\Big(\frac{1-|z|^2}{1-|w|^2}\Big)^{\beta} \ dw $$
    An operator $S$ is $\beta$-weakly localized if the following conditions are satisfied:
    \begin{enumerate}
        \item $\I{S}(0)<\infty$
        \item $\I{S^*}(0)<\infty$
        \item $\lim_{r\to 1^-}\I{S}(r)=0$
        \item $\lim_{r\to 1^-}\I{S^*}(r)=0$.
    \end{enumerate}
\end{definition}
We denote the algebra of $\beta$-weakly localized operators by $\wloc$.

Given an operator $S\in \bdd$ and $O\subset G$, we formally define $S_O$ by the weak-integral
\begin{align}\label{eq:SO}
    S_O=\int_{G\times G} \chi_{O}(g) \ip{\actopL{h^{-1}}{S}1}{\pi(g)1}\  \actopL{h}{(\pi(g)1)\otimes 1} \ d\mu_{G\times G}(g,h).
\end{align}
Note that the above integral w.r.t. the product measure may not exist. 

For $r\in(0,1)$, we define $$\Br=\{g\in G\mid |g\cdot0|<r \}.$$

\subsection{The operator \texorpdfstring{$S_{\Br}$}{Br}}
Here, we prove that for any $S\in \bdd$, the operator $S_{\Br}$ exists and can be approximated by Toeplitz operators. 
In the following lemma, we prove that $S_O$ exists when $O$ has finite Haar measure.
\begin{lemma}\label{lem:finitemeasure}
    Let $S\in \bdd$ and let $O\subset G$ s.t. $\mu_G(O)<\infty$. Then $S_{O}$ exists as a bounded operator. Moreover, we can express $S_{O}$ as
    $$S_{O}=\int_{O} a_{S,g} \ast \Phi_g \haar{g}.$$
    where $a_{S,g}\in \bddfG$ is given by
    $$a_{S,g}(h)=\ip{\actopL{h^{-1}}{S}1}{\pi(g)1},\ \ h\in G$$
    and $\Phi_g=\pi(g)\otimes 1$, for $g\in G$.
\end{lemma}

\begin{proof}
    Note that for $f_1,f_2\in\Berg$.
    \begin{align*}
        \int_O\int_G & |\ip{\actopL{h^{-1}}{S}1}{\pi(g)1}\  \ip{\actopL{h}{(\pi(g)1)\otimes 1}f_1}{f_2}| \haar{h}\haar{g}\\
        &\leq \|S\| \int_O\int_G |\ip{\actopL{h}{(\pi(g)1)\otimes 1}f_1}{f_2}| \haar{h}\haar{g}\\
        &= \|S\| \int_O\int_G |\ip{\pi(h^{-1})f_1}{1}\ip{\pi(g)1}{\pi(h^{-1})f_2}| \haar{h}\haar{g}\\
        &= \|S\| \int_O\int_G |\ip{\pi(h)f_1}{1}\ip{\pi(g)1}{\pi(h)f_2}| \haar{h}\haar{g}\\
        &\leq \|S\| \int_O \|\pi_{1,f_1}\|_2 \|\pi_{\pi(g)1,f_2}\|_2 \haar{g}\\
        &= \|S\|\int_O \|1\|\|f_1\|\|\pi(g)1\|\|f_2\|  \haar{g} \\
        &\hspace{4cm}\text{(by Schur-orthogonality relations)}\\
        &=\|S\|\mu_G(O) \|f_1\|\|f_2\| .
    \end{align*}
    Hence, by Fubini's theorem, $S_O$ is bounded with $\|S_O\|\leq \|S\| \mu_G(O)$. 
\end{proof}

Now we make the observation, that $S_{\Br}$ is an integral of QHA convolutions. This is indeed the fact that leads to our concise proof. 

\begin{proposition}
    Let $S\in \bdd$ and $r\in[0,1)$. Then $S_{\Br}$ exists as a bounded operator. Moreover, we can express $S_{\Br}$ as
    $$S_{\Br}=\int_{|z|<r} a_{S,z} \ast \Phi_z \inv{z}.$$
    where $a_{S,z}\in \bddfG$ is given by
    $$a_{S,z}(h)=\ip{\actopL{h^{-1}}{S}1}{k_{z}},\ \ h\in G$$
    and $\Phi_z=k_{z}\otimes 1$, for $z\in \B^n$.
\end{proposition}
\begin{proof}
    First, note that $S_{\Br}\in \bdd$ by Lemma \ref{lem:finitemeasure}, as
\begin{align*}
    \mu_G(\Br)&=\int_{G}\chi_{\Br}(g) 1\haar{g}\\
    &= \int_{|z|<r} \frac{1}{(1-|z|^2)^{n+1}} \ dz \\
    &<\infty.
\end{align*}
We start with the warning that the pair of functions $a_{S,g}$ and $a_{S,g\cdot 0}$, and the pair of operators $\Phi_g$ and $\Phi_{g\cdot 0}$, are not interchangeable. However, by \ref{eq:cocyclerep}, we have that
\begin{align*}
    a_{S,g}\ast \Phi_g&=  \int_G |j(g,0)|^2 \ip{\actopL{h^{-1}}{S}1}{K_{g\cdot 0}}\  \actopL{h}{(K_{g\cdot 0})\otimes 1} \haar{h}\\
    &= \int_G \ip{\actopL{h^{-1}}{S}1}{k_{g\cdot 0}}\  \actopL{h}{(k_{g\cdot 0})\otimes 1} \haar{h}\\
    &= a_{S,g\cdot 0}\ast \Phi_{g\cdot 0}.
\end{align*}

Hence, it follows that
\begin{align*}
    S_{\Br}
    &=\int_{\Br} a_{S,g\cdot 0} \ast \Phi_{g\cdot 0} \haar{g}\\
    &=\int_{|z|<r} \int_G \ip{\actopL{h^{-1}}{S}1}{k_{z}}\  \actopL{h}{k_{z}\otimes 1} \haar{h} \inv{z}\ \ \text{(by \ref{eq:integral})}\\
    &= \int_{|z|<r} \int_G a_{S,z}(h) \actopL{h}{\Phi_z} \haar{h} \inv{z}\\
    &=\int_{|z|<r} a_{S,z} \ast \Phi_z \inv{z}.\qedhere
\end{align*}
\end{proof}

In \cite{BC94}, Berger and Coburn prove that the set of all Toeplitz operators with symbols in $\Lone$ forms a dense subset of $\traceop$. This can also be viewed as a consequence of the QHA Wiener's Tauberian theorem discussed in \cite{DDMO24}. We use this fact in the following lemma.

\begin{lemma}\label{lem:convol_traceclass}
    We have the following inclusion
    $$\bddfG\ast \traceop\subset \overline{\{T_a\mid a\in \bddf\}}.$$
\end{lemma}

\begin{proof}
    Let $\psi \in \bddfG$, $S\in \traceop$ and let $\epsilon>0$. Since Toeplitz operators are dense in $\traceop$, there is $a\in \Lone$ s.t. 
    $$\|T_a-S\|<\frac{\epsilon}{\|\psi\|_\infty+1}.$$
    Then
    $$\|\psi\ast T_a-\psi\ast S\|=\|\psi\ast (T_a-S)\|\leq \|\psi\|_\infty\|T_a-S\|<\epsilon.$$
    Also,
    $$\psi\ast T_a=\psi\ast(a\ast \Phi)=(\psi\ast a)\ast \Phi=T_{\psi\ast a}.$$
    Finally, we note that $\psi\ast a \in \bddf$. This completes the proof. 
\end{proof}

\begin{proposition}\label{prop:SBr_Toeplitz}
    Let $S\in \bdd$ and let $r\in(0,1)$. Then $S_{\Br}$ exists and $S_{\Br}\in \overline{\{T_a\mid a\in \bddf\}}$. 
\end{proposition}

\begin{proof}
    Since $a_{S,z}\ast \Phi_z\in \bddfG\ast \traceop$, we have that 
    $$a_{S,z}\ast \Phi_z\in \overline{\{T_a\mid a\in \bddf\}}$$ by Lemma \ref{lem:convol_traceclass}.
    Now we claim that the map $\B^n\to \bdd,\ z\mapsto a_{S,z}\ast \Phi_z$ is continuous w.r.t. the operator norm. Since $\B^n$ is the homogeneous space $G/K$, and because $a_{S,g}\ast \Phi_{g}=a_{S,g\cdot 0}\ast \Phi_{g\cdot 0}$ it suffices to prove that $g\mapsto a_{S,g}\ast \Phi_{g}$ is continuous from $G\to \bdd$ w.r.t. the operator norm.
    
    Note that for $g,h\in \B^n$,
    \begin{align*}
        \|a_{S,g}\ast \Phi_g-a_{S,h}\ast \Phi_h\|
        &= \|a_{S,g}\ast \Phi_g-a_{S,h}\ast \Phi_g\|
        + \|a_{S,h}\ast \Phi_g-a_{S,h}\ast \Phi_h\|\\
        &=\|(a_{S,g}-a_{S,h})\ast \Phi_g\|
        +\|a_{S,h}\ast (\Phi_g-\Phi_h)\|\\
        &\leq \|a_{S,g}-a_{S,h}\|_\infty\|\Phi_g\|_1 +\|a_{S,h}\|_\infty \|\Phi_g-\Phi_h\|_1\\
        &\leq 2\|S\|\|\pi(g)1-\pi(h)1\|.
    \end{align*}
    By the strong continuity of the representation the required continuity holds. Then as $S_{\Br}$ is a limit of the Reimann sums of the operators of the form $\frac{1}{(1-|z|^2)^{n+1}}a_{S,z}\ast \Phi_z$, we have that
     $S_{\Br}$ is also in $\overline{\{T_a\mid a\in \bddf\}}$. 
\end{proof}

Lastly, we write $S_{\Br}$ as a double integrals over the ball:
\begin{lemma}\label{lem:SBrintegral}
    Let $S\in \bdd$ and $r\in[0,1)$. Then
    $$S_{\Br}=\int_{\B^n}\int_{\Dr(z)} \ip{SK_z}{K_w}  (K_w\otimes K_z) \ dw\ dz.$$
\end{lemma}

\begin{proof}
    We already noted that $S_{\Br}\in \bdd$. Note that
    \begin{align*}
        S_{\Br}&= \int_{\Br}\int_G \ip{\actopL{h^{-1}}{S}1}{\pi(g)1}\  \actopL{h}{\pi(g)1\otimes 1} \haar{h}\haar{g}\\
        &=\int_G\int_G \chi_{\Br}(g) \ip{\actopL{h^{-1}}{S}1}{\pi(g)1}\  \actopL{h}{\pi(g)1\otimes 1} \haar{g}\haar{h}\\
        &\hspace{6cm}\text{(by a Fubini's theorem)}\\
        &=\int_G\int_G \chi_{\Br}(h^{-1}g)\ip{\actopL{h^{-1}}{S}1}{\pi(h^{-1}g)1}\\
        &\hspace{4.5cm}\times \actopL{h}{(\pi(h^{-1}g)1)\otimes 1} \haar{g}\haar{h}\\
        &\hspace{6cm}\text{(by a change of variable)}\\
        &=\int_G\int_G \chi_{\Br}(h^{-1}g) \ip{S\pi(h)1}{\pi(g)1}\  \pi(g)1\otimes \pi(h)1 \haar{g}\haar{h}.
    \end{align*}
    Recall that $\pi(g)1=j(g,0)K_{g\cdot 0}$ by \ref{eq:cocyclerep} and $|h^{-1}g|=|\act{h\cdot 0}{g\cdot 0}|$ by Lemma \ref{lem:absolute_value}. Hence we have
    \begin{align*}
         S_{\Br}
         &= \int_G \int_G \chi_{\Br}(h^{-1}g) \ip{SK_{h\cdot 0}}{K_{g\cdot 0}} ( K_{g\cdot 0}\otimes K_{h\cdot 0}) |j(g,0)|^2\haar{g}\\
         &\hspace{7cm}\times|j(h,0)|^2\haar{h}\\
         &= \int_{\B^n}\int_{\Dr(z)} |\ip{SK_z}{K_w}  \ (K_w\otimes K_z) \ dw\ dz \ \ \text{(by \ref{eq:integraldz})}.\qedhere
    \end{align*}
\end{proof}

\subsection{An integral representation of weakly localized operators}
In \cite{WZ23}, the authors discuss an integral representation of weakly localized operators, which we present here for completeness.
Namely, we prove that a weakly localized operator $S\in \wloc$ has the integral representation $S=S_G$ and $\lim_{r\to 1^-}S_{\Br}=S$ in operator norm. 

The following integral representation holds for any bounded operator.
\begin{lemma}\label{lem:integral_S}
    Let $S\in \bdd$. Then
    $$S=\int_G\int_G \ip{S\pi(h)1}{\pi(g)1}\  \pi(g)1\otimes \pi(h)1 \haar{g}\haar{h}.$$
\end{lemma}

\begin{proof}
    The identity operator is the Toeplitz operator whose symbol is the constant function one. Hence
    $$S=\id S \id= T_1 S T_1=(1\ast \Phi)S(1\ast \Phi).$$
    Then we have
    \begin{align*}
        S&=\int_G (1\ast \Phi)S\actopL{g}{\Phi} \haar{g}\\
        &= \int_G \int_G \actopL{h}{\Phi}S\actopL{g}{\Phi} \haar{g}\\
        &= \int_G \int_G (\pi(h)1\otimes\pi(h)1)S(\pi(g)1\otimes\pi(g)1) \haar{g}\haar{h}\\
        &= \int_G\int_G \ip{S\pi(h)1}{\pi(g)1}\  \pi(g)1\otimes \pi(h)1 \haar{g}\haar{h}.\qedhere
    \end{align*}
\end{proof}

The following Proposition was discussed in \cite[Theorem 3.3]{WZ23} and we include this only for completeness. We also note that the Lemma 3.4 in \cite{X15}, is also of the same nature. We use the notations, $\Br^\mathsf{c}=G\setminus \Br$ and $\Dr(z)^\mathsf{c}=\B^n\setminus \Dr(z)$. 

\begin{proposition}\label{prop:SBr_convergence}
    Let $S\in\wloc$. Then $S_G$ exists as a bounded operator and we have $S_G=S$. Moreover,
    $\lim_{r\to 1^-} S_{\Br}=S$
    in operator norm. 
\end{proposition}

\begin{proof}
    Let $f_1,f_2\in \Berg$. First we prove that $\lim_{r\to 1^-} S_{\Br^\mathsf{c}}=0$ in operator norm. Note that 
    \begin{align*}
        &\int_{D_r(z)^\mathsf{c}} |\ip{SK_z}{K_w}  f_2(w)| \ dw\\
        &= \int_{D_r(z)^\mathsf{c}}  |\ip{SK_z}{K_w}|^{\frac{1}{2}} \Big(\frac{1-|z|^2}{1-|w|^2} \Big)^{\frac{\beta}{2}}   |\ip{SK_z}{K_w}|^{\frac{1}{2}} |f_2(w)|\Big(\frac{1-|w|^2}{1-|z|^2} \Big)^{\frac{\beta}{2}} \ dw\\
        &\leq \Bigg(\int_{D_r(z)^\mathsf{c}}  |\ip{SK_z}{K_w}| \Big(\frac{1-|z|^2}{1-|w|^2} \Big)^{\beta} \ dw\Bigg)^{\frac{1}{2}}\\
        &\hspace{3.5cm}\times \Bigg(\int_{D_r(z)^\mathsf{c}}  |\ip{SK_z}{K_w}| |f_2(w)|^2\Big(\frac{1-|w|^2}{1-|z|^2} \Big)^{\beta} \ dw\Bigg)^{\frac{1}{2}}\\
        &\hspace{6cm} \text{( by Holder's inequality)}\\
        &= \I{S}(r)\Bigg(\int_{D_r(z)^\mathsf{c}}  |\ip{SK_z}{K_w}| |f_2(w)|^2\Big(\frac{1-|w|^2}{1-|z|^2} \Big)^{\beta} \ dw\Bigg)^{\frac{1}{2}}.
    \end{align*}

    Now by Lemma \ref{lem:SBrintegral}, we have
    \begin{align*}
        \int_{\B^n}&\int_{D_r(z)^\mathsf{c}} |\ip{SK_z}{K_w}  \ \ip{(K_w\otimes K_z)f_1}{f_2}| \ dw\ dz \\
        &= \int_{\B^n} |f_1(z)| \int_{D_r(z)^\mathsf{c}} |\ip{SK_z}{K_w}  f_2(w)| \ dw \ dz\\
        &\leq \|f_1\|\Bigg( \int_{\B^n} \bigg(\int_{D_r(z)^\mathsf{c}} |\ip{SK_z}{K_w}  f_2(w)| \ dw \bigg)^2 \ dz\Bigg)^\frac{1}{2}\\
        &\hspace{4cm}\text{(by Holder's inequality)}\\
        &\leq \|f_1\| \I{S}(r) \Bigg( \int_{\B^n} \int_{D_r(z)^\mathsf{c}}  |\ip{SK_z}{K_w}| |f_2(w)|^2\Big(\frac{1-|w|^2}{1-|z|^2} \Big)^{\beta} \ dw \ dz\Bigg)^\frac{1}{2}\\
        &\hspace{7cm}\\
        &= \|f_1\| \I{S}(r) \Bigg( \int_{\B^n} |f_2(w)|^2\int_{D_r(w)^\mathsf{c}}  |\ip{K_z}{S^*K_w}| \Big(\frac{1-|w|^2}{1-|z|^2} \Big)^{\beta} \ dz \ dw\Bigg)^\frac{1}{2}\\
        &\leq \I{S}(r)\I{S^*}(r)\|f_1\| \|f_2\| 
    \end{align*}
    By taking $r=0$, we get that $S_G\in \bdd$ with $\|S_G\|\leq \I{S}(0)\I{S^*}(0)$.     
    Also, we have
    \begin{align*}
        S_{\Br^\mathsf{c}}&=S_G-S_{\Br}\\
        &= \int_{\B^n}\int_{\Dr(z)^\mathsf{c}} \ip{SK_z}{K_w}  (K_w\otimes K_z) \ dw\ dz.
    \end{align*}
     by Lemma \ref{lem:SBrintegral}.
    Therefore we get $\| S_{\Br^\mathsf{c}}\|\leq \I{S}(r)\I{S^*}(r).$
    It follows that
    $\lim_{r\to 1^-} S_{\Br^\mathsf{c}}=0$
    in operator norm.
    Also, by Lemma \ref{lem:integral_S}, we have $S=S_G=S_{\Br}+S_{\Br^\mathsf{c}}$.
    Therefore we have  $\lim_{r\to 1^-} S_{\Br}=S$
    in operator norm. 
    \end{proof}

    We now prove that weakly localized operators are contained in the norm closure of the set of all Toeplitz operators $\toep$.

    \begin{proposition}
        We have the following
        $$\wloc\subset \overline{\toep}.$$
    \end{proposition}
    \begin{proof}
        This follows immediately by Proposition \ref{prop:SBr_Toeplitz} and Proposition \ref{prop:SBr_convergence}.
    \end{proof}

    Now we reprove Xia's theorem. We also add the additional statement
    $$\toepalg=\overline{\bddfG\ast \traceop}.$$
    
    \begin{theorem}\label{theo:wloc_Toeplitz}
        We have the following:
        $$\toepalg=\overline{\wloc}=\overline{\{T_a\mid a\in \bddf\}}=\overline{\bddfG\ast \traceop}.$$
    \end{theorem}

    \begin{proof}
        Note that we have 
        $$\{T_a\mid a\in \bddf\}\subset \wloc\subset \overline{\{T_a\mid a\in \bddf\}}$$
        Thus
        $$\overline{\wloc}=\overline{\{T_a\mid a\in \bddf\}}\subset\toepalg.$$
        Also, by the minimality of $\toepalg$, we have $\toepalg\subset \overline{\wloc}$. Then
        $\toepalg= \overline{\wloc}$ as needed. We also have
        $$\{T_a\mid a\in \bdd\}\subset \bddfG\ast \traceop\subset \toepalg=\overline{\{T_a\mid a\in \bddf\}},$$
        Then the equality $\toepalg=\overline{\bddfG\ast \traceop}$  holds.
    \end{proof}

\bibliographystyle{plain} 
\bibliography{bibi}

\end{document}